\documentclass[english]{amsart}

\usepackage{amsmath,amssymb,enumerate,bbm}

\usepackage[T1]{fontenc}
\usepackage[utf8]{inputenc}
\usepackage[all]{xy}

\usepackage{babel}
\usepackage{amstext}
\usepackage{amsmath}
\usepackage{amsfonts}
\usepackage{latexsym}
\usepackage{ifthen}

\usepackage{xypic}
\xyoption{all}
\newtheorem{lemma1}[equation]{}

\newenvironment{lemma}{\begin{lemma1}{\bf Lemma.}}{\end{lemma1}}
\newenvironment{example}{\begin{lemma1}{\bf Example.}\rm}{\end{lemma1}}

\newenvironment{theorem}{\begin{lemma1}{\bf Theorem.}}{\end{lemma1}}
\newenvironment{proposition}{\begin{lemma1}{\bf Proposition.}}{\end{lemma1}}
\newenvironment{corollary}{\begin{lemma1}{\bf Corollary.}}{\end{lemma1}}
\newenvironment{remark}{\begin{lemma1}{\bf Remark.}\rm}{\end{lemma1}}

\newenvironment{question}{\begin{lemma1}{\bf Question.}}{\end{lemma1}}

\newenvironment{remark*}{{\bf Remark.}}{}
\newenvironment{example*}{{\bf Example.}}{}

\newcommand{\Q}{\ensuremath{\mathbb{Q}}}

\newcommand{\PP}{\ensuremath{\mathbb{P}}}

\newcommand{\holom}[3]{\ensuremath{#1:#2  \rightarrow #3}}
\newcommand{\fibre}[2]{\ensuremath{#1^{-1} (#2)}}

\makeatletter
\ifnum\@ptsize=0  \fi \ifnum\@ptsize=2  \fi 

\newcommand\sO{{\mathcal O}}

\DeclareMathOperator*{\sing}{sing}

\title{Anticanonical divisors and curve classes   \\ on Fano manifolds}
\date{September 14, 2010}

\author{Andreas H\"oring}
\address{Andreas H\"oring, Universit\'{e} Pierre et Marie Curie and Albert-Ludwig Universit\"at Freiburg}
\curraddr{Mathematisches Institut, Albert-Ludwigs-Universit\"at
  Freiburg, Eckerstra{\ss}e 1, 79104 Freiburg im Breisgau, Germany}
\email{hoering@math.jussieu.fr}

\author{Claire Voisin}
\address{Claire Voisin, CNRS, Universit\'{e} Pierre et Marie Curie, Institut de Math\'{e}matiques de Jussieu,
TGA Case 247, 4 place Jussieu, 75005 Paris, France}
\email{voisin@math.jussieu.fr}

\begin{document}

\maketitle

\vspace{-0.5cm}
\begin{center}
{\em \small
To Eckart Viehweg
}
\end{center}

\begin{abstract}
It is well known that the Hodge conjecture with rational coefficients holds for degree
$2n-2$ classes on complex projective $n$-folds.
In this paper we study the more precise question if on a rationally connected
complex projective $n$-fold the {\it integral} Hodge classes of degree $2n-2$ are generated over $\mathbb Z$
by classes of curves.
We combine techniques from the theory of singularities of pairs on the one hand  and
infinitesimal variation of Hodge structures on the other hand to give
an affirmative answer to this question for a large class
of manifolds including Fano fourfolds. In the last case, one step in the proof
is the following result of independent interest: There exist anticanonical divisors  with isolated canonical singularities
on a smooth Fano fourfold.

\end{abstract}

\vspace{-1ex}
\section{Introduction}
The Hodge conjecture states that the space $Hdg^{2i}(X)$
of rational Hodge classes on a smooth projective complex variety
$X$ is generated over $\mathbb{Q}$ by classes of algebraic cycles of codimension $i$ on $X$.

This conjecture is  known to be wrong for integral coefficients instead of rational coefficients
(\cite{ah}, \cite{kollar}). We will focus in this paper
on Hodge classes of degree $2n-2,\,n=dim\,X$ (or ``curve classes'') for which the Hodge conjecture is known
to hold with rational coefficients. As remarked in \cite{voisinsoule}, this case
is particularly interesting as it leads to a birational invariant of $X$, namely
the finite group $$Z^{2n-2}(X):=Hdg^{2n-2}(X,\mathbb{Z})/\langle[Z],\,codim\,Z=n-1\rangle,$$ where
$$Hdg^{2n-2}(X,\mathbb{Z}):=\{\alpha\in H^{2n-2}(X,\mathbb{Z}),\,\alpha_\mathbb{C}\in H^{n-1,n-1}(X)\}.$$
 Koll\'ar's counterexamples show that this group can be non trivial starting from
 $n=3$ and are  strikingly simple:
He considers hypersurfaces $X$ of degree $d$ in $\mathbb{P}^4$. The cohomology
group $H^4(X,\mathbb{Z})$ is then isomorphic to $\mathbb{Z}$, with generator $\alpha$ of degree $1$ with respect to the hyperplane class.
\begin{theorem}(Koll\'ar, \cite{kollar}) Assume that for some  integer
$p$ coprime to $6$, $p^3$ divides $d$. Then for very general such $X$, any curve $C\subset X$ has
degree divisible by $p$. Hence the class $\alpha$ is not algebraic,  (that is, is not the class of an algebraic cycle).
\end{theorem}
Koll\'ar's examples are hypersurfaces of general type (and in fact, hypersurfaces
with negative or trivial canonical class contain a line, whose cohomology class
is the generator $\alpha$). It was proved more generally in \cite{voisinuniruled} that the Kodaira dimension
plays an important role in the study of the group $Z^4(X)$, when $dim\,X=3$:
\begin{theorem} (Voisin 2006) \label{theoremvoisin} Let $X$ be a smooth projective threefold. Assume that $X$ is
either uniruled or Calabi-Yau. Then the Hodge conjecture holds for integral Hodge classes on $X$, that is, the
group $Z^4(X)$ is trivial.
\end{theorem}
The assumptions on the Kodaira dimension are ``essentially optimal''. (Note however that in this theorem, the assumption
in the Calabi-Yau case is
not actually an assumption on the Kodaira dimension, since we do not know whether,
 in the theorem above, it could be replaced by the weaker assumption that
the Kodaira dimension is $0$.)
Indeed, mimicking Starr's degeneration argument in \cite{starr}, one can show that a
very
general hypersurface
$X$ in $\mathbb{P}^1\times \mathbb{P}^3$ of bidegree $(3,4)$, which has Kodaira dimension $1$,
has a non trivial
group $Z^4(X)$: More precisely, Lefschetz hyperplane section theorem shows that the group
$H^4(X,\mathbb{Z})=H_2(X,\mathbb{Z})$ consists of Hodge classes and surjects via the Gysin map $pr_{1*}$ onto
$\mathbb{Z}=H_2(\mathbb{P}^1,\mathbb{Z})$. On the other hand, one can show that for very general
such
$X$,
any curve $C\subset X$ has even degree over $\mathbb{P}^1$ so that the images
via the Gysin map $pr_{1*}$ of algebraic classes do not generate
$H_2(\mathbb{P}^1,\mathbb{Z})$  over $\mathbb{Z}$.

In  dimension $4$ and higher, the uniruledness assumption does not say anything on the group
$Z^{2n-2}(X)$. Indeed, start from one of the Koll\'ar's examples $X_1$, and consider
the uniruled fourfold
$X:=X_1\times \mathbb{P}^1$. The degree $6$ integral Hodge class
$$\beta:=pr_1^*\alpha\cup pr_2^*[pt],$$
is not algebraic on $X$, since otherwise $pr_{2*}\beta=\alpha$ would be algebraic on
$X_1$.

The following question  was raised in \cite[section 2]{voisinsoule}:

 \begin{question} \label{qvs} Is the integral cohomology of degree $2n-2$ of a rationally connected manifold of dimension
 $n$ generated over $\mathbb{Z}$ by classes of curves?
 \end{question}

We will prove in this paper the following result, which answers positively this question for certain Fano manifolds, and also for certain projective manifolds $X$ for which $-K_X$ is only assumed to be $1$-ample. Recall from
\cite{sommese} that a vector bundle $\mathcal{E}$ on $X$ is said to be $1$-ample if
some multiple $\mathcal{O}_{\PP(E)}(l), \,l>0$ of the tautological line bundle
on $\mathbb{P}(\mathcal{E})$ is globally generated, and the fibers of the morphism
$\mathbb{P}(\mathcal{E})\rightarrow \mathbb{P}^N$ given by sections of
$\mathcal{O}_{\PP(E)}(l)$ are at most $1$-dimensional.

\begin{theorem}\label{theoclasses}  Let $X$ be a rationally connected, projective manifold of dimension $n \geq 4$.
Assume that there exists a $1$-ample vector bundle $\mathcal{E}$ of rank $n-3$ on $X$ such that $det\,\mathcal{E}=-K_X$ and
$\mathcal{E}$ has a transverse
section
whose zero set $Y$ has isolated canonical singularities.  Then   the group
$ H^{2n-2}(X,\mathbb{Z})$ is generated over $\mathbb{Z}$ by classes of curves (equivalently, $Z^{2n-2}(X)=0$).
\end{theorem}
The property that the $3$-fold $Y$ has isolated singularities will be crucial for the proof
of this Theorem :
it assures that if we consider a sufficiently general ample divisor $\Sigma \subset Y$, it does not meet
the singular locus of $Y$. We then generalize the methods from \cite{voisinuniruled},
based on the study of the infinitesimal variation of Hodge structure of these
surfaces $\Sigma$, to conclude.

If $\mathcal{E}$ is globally generated,  a generic section of
$\mathcal{E}$ is transverse with smooth vanishing locus; hence we get the following corollary:

\begin{corollary}  \label{corollarygenerated}
Let $X$ be a rationally connected, projective manifold of dimension $n \geq 4$.
Assume that  there exists a globally generated $1$-ample vector bundle $\mathcal{E}$ of rank $n-3$ on $X$ such that $det\,\mathcal{E}=-K_X$. Then   the group
$ H^{2n-2}(X,\mathbb{Z})$ is generated over $\mathbb{Z}$ by classes of curves.
\end{corollary}

Let us now give a class of examples covered by Theorem \ref{theoclasses}, but not by Corollary \ref{corollarygenerated}:
a Fano manifold of dimension $n \geq 4$ has index $n-3$ if there exists an ample Cartier divisor $H$ (called the fundamental
divisor) such that $-K_X \simeq (n-3) H$.  For these manifolds we
will prove the following result:

\begin{theorem} \label{theoremisolated}
Let $X$ be a smooth Fano manifold of dimension $n \geq 4$ and index $n-3$.
Suppose moreover that the fundamental divisor satisfies $h^0(X, H) \geq n-2$.
Let $Y$ be the zero locus of a general section of the vector bundle $\mathcal{E}=H^{\oplus n-3}$.
Then $Y$  has dimension $3$ (i.e. the general section is transverse) and has
isolated canonical singularities.
\end{theorem}

We expect the technical condition $h^0(X, H) \geq n-2$ to be automatically satisfied and prove
this in low dimension  (cf. also \cite{Flo10}).
In particular we obtain the following unconditional theorem concerning Question
\ref{qvs} in the case of Fano manifolds of dimension at most five.

\begin{theorem} \label{theoremfourfolds}
i) Let $X$ be a smooth Fano fourfold. Then a general anticanonical divisor $Y \in |-K_X|$
has canonical isolated singularities.

ii)  $X$ being as above,  the group
$H^{6}(X,\mathbb{Z})$ is generated over $\mathbb{Z}$ by classes of curves (equivalently, $Z^{6}(X)=0$).

iii) Let $X$ be a smooth Fano fivefold of index $2$. Then the group
$H^{8}(X,\mathbb{Z})$ is generated over $\mathbb{Z}$ by classes of curves (equivalently, $Z^{8}(X)=0$).
\end{theorem}

\begin{remark}  According to Koll\'ar \cite{kollarportugaliae}, Iskovskikh asked
 whether curve classes on  Fano manifolds $X$ are generated over $\mathbb{Z}$
 by classes of  {\it rational} curves. This question is of a different nature and leads to another birationally invariant group associated to $X$, namely
 the group of curve classes on $X$ modulo the subgroup generated by classes of
 rational curves in $X$. Koll\'ar (see \cite[Thm.3.13]{kollarbook}) proves
 that for rationally connected manifolds $X$, curve classes on   $X$ are generated over $\mathbb{Q}$
 by classes of  rational  curves. In other words, the group introduced above
 is of torsion.
The method used here does not give any insight on the Iskovskikh  question, that is whether this torsion group is zero or not for Fano manifolds.

 \end{remark}
Note that our result  in Theorem \ref{theoremfourfolds}, i)
on the anticanonical divisor is optimal, i.e.
it is easy to construct examples of Fano fourfolds without  any smooth
anticanonical divisor   (cf. Example \ref{examplenotsmooth}).
In particular it is not possible to prove the  statement ii) of Theorem \ref{theoremfourfolds}
simply by applying
the Lefschetz hyperplane theorem to $Y$ and using  the Calabi-Yau case of Theorem \ref{theoremvoisin}.

Let us explain the proof  of Theorem \ref{theoremisolated}  if $X$ is a fourfold.
Based on subadjunction techniques and effective non-vanishing results, Kawamata has shown
in \cite{Kaw00} that a general anticanonical divisor $Y$ is a Calabi-Yau threefold with at most canonical singularities,
in particular the singular locus of $Y$ has dimension at most one.
We will argue by contradiction and suppose that a general $Y$ is singular along a curve $C$.
A priori the curve $C$ depends on $Y$, but we show that if we fix a pencil spanned by two general elements $Y_1, Y_2$  in $|-K_X|$,
all the elements of the pencil are singular along the same curve $C$.
In particular the pair $(X, Y_1+Y_2)$ is not log-canonical along the curve $C$, by inversion of adjunction
this contradicts a result of Kawamata on ample divisors on Calabi-Yau threefolds \cite[Prop.4.2]{Kaw00}.

If we apply the same argument to Fano threefolds we obtain a new proof of Shokurov's theorem.

\begin{theorem} \label{theoremshokurov} \cite{Sho79}
Let $X$ be a smooth Fano threefold.
If $Y \in |-K_X|$ is a general element, it is smooth.
\end{theorem}

Earlier proofs of this statement used Saint Donat's study of linear systems on K3 surfaces \cite{SDo74}
which so far has no analogue in higher dimension.
By contrast our strategy of proof generalizes immediately to arbitrary dimension and shows the following:
if Kawamata's nonvanishing conjecture \cite[Conj.2.1]{Kaw00} is true,
the singular locus of a general anticanonical divisor on any Fano manifold has codimension at least three (which is the optimal bound).

Let us conclude this introduction with a comment on assumptions in Theorem \ref{theoclasses}:
in the proof of Theorem \ref{theoclasses} we will only use that $-K_X$ is $1$-ample and
the cohomology group  $H^1(X,\mathcal{O}_X)$ vanishes. This seems to be weaker than supposing that $X$ is rationally connected and $-K_X$ is $1$-ample,
but in fact these conditions are equivalent:

\begin{proposition}\label{newprop} Let $X$ be a projective manifold such that
$-K_X$ is $1$-ample and $H^1(X,\mathcal{O}_X)=0$. Then $X$ is rationally connected.
\end{proposition}

\begin{proof}
Since $-K_X$ is nef with numerical dimension $n-1$, we have $H^i(X,\mathcal{O}_X)=0$, for $i>1$ \cite[6.13]{Dem00}.
 Hence we have $\chi(X,\mathcal{O}_X)=1$. Furthermore $X$ has no finite \'{e}tale cover, because
an  \'{e}tale cover $X'\rightarrow X$ of degree $m>1$ would also satisfy the condition
 that $-K_{X'}$ is $1$-ample, which implies
as above that $H^i(X',\mathcal{O}_{X'})=0$, for $i>1$. Thus $\chi(X',\mathcal{O}_{X'})\leq 1$, which contradicts $\chi(X',\mathcal{O}_{X'})=m\chi(X,\mathcal{O}_X)=m$.

 According to \cite{dps}, it follows that $X$ is a product  of Calabi-Yau
 varieties, symplectic holomorphic varieties and varieties $W$ such that
 $H^0(W,\Omega_W^{\otimes l})=0,\,l>0$. The first two types can not occur because
 $H^0(X,\Omega_X^l)=0,\,l>0$.
 It follows that $H^0(X,\Omega_X^{\otimes l})=0,\,l>0$.

 Consider now the maximal rationally connected fibration $\psi:X\dashrightarrow B$ of $X$
 (cf. \cite{kmm}). If $X$ is not rationally connected, $B$ is not uniruled by \cite{ghs},
 and thus the Kodaira dimension of $B$ is $0$, according to \cite{dQZ}. But then, if $b:={\rm dim}\,B>0$,
 there is a non zero section of $K_B^{\otimes m}$ for some $m>0$, which gives a non zero section
 of $\Omega_X^{\otimes bm}$ and provides a contradiction.
\end{proof}

\vspace{0.5cm}

{\it We dedicate this paper  to the memory of  Eckart Viehweg, who contributed in a major way to
both birational geometry and Hodge theory.}

\section{Anticanonical divisors}

While anticanonical divisors on Fano threefolds (and Fano $n$-folds of index $n-2$)
are well understood \cite{Sho79, Mel}, our knowledge on higher-dimensional Fano manifolds remains rather limited.
A first step was done by Kawamata:

\begin{theorem} \cite[Thm.5.2]{Kaw00} \label{theoremkawamata}
Let $X$ be a Fano fourfold with at most canonical Gorenstein singularities.
If $Y \in |-K_X|$ is a general element, it is a Calabi-Yau threefold with at most canonical singularities.
\end{theorem}

In view of this statement and Shokurov's theorem \ref{theoremshokurov}
one might hope that for a smooth Fano fourfold
the general anticanonical divisor $Y$ is also smooth. Here is an easy counterexample which shows that
such an $Y$ might even not be $\mathbb{Q}$-Cartier :

\begin{example} \label{examplenotsmooth}
Let $S$ be the blow-up of $\PP^2$ in eight points in general position. Then $S$ is a Fano surface
whose anticanonical system has exactly one base point which we denote by $p$.
Set $X:=S \times S$ and $S_i:=\fibre{p_i}{p}$ where $p_i$ is the projection on the $i$-th factor.
Then $X$ is a smooth Fano fourfold and
\[
Bs |-K_X| = S_1 \cup S_2.
\]
Let $Y \in |-K_X|$ be a general element, then
$Bs |-K_X| \subset Y$, so the surfaces $S_1, S_2$ are Weil divisors in $Y$.
If they were $\Q$-Cartier, their intersection $S_1 \cap S_2$ would have dimension at least one,
yet we have $S_1 \cap S_2=(p,p)$.
Thus the variety $Y$ is not $\Q$-factorial.
\end{example}

Very recently Floris generalised Theorem \ref{theoremkawamata} to Fano varieties of index $n-3$:

\begin{theorem} \cite{Flo10} \label{theoremfloris}
Let $X$ be a smooth Fano manifold of dimension $n \geq 4$ and index $n-3$.
Suppose moreover that the fundamental divisor satisfies $h^0(X, H) \geq n-2$.
Then there exists a sequence
\[
X \supsetneq Z_1 \supsetneq  Z_2 \supsetneq \ldots \supsetneq Z_{n-3}
\]
such that $Z_{i+1} \in |H|_{Z_i}|$ is a normal variety of dimension $n-i$ with at most canonical singularities.
\end{theorem}

\subsection{Proof of the structure results}

Let us recall that if $X$ is a normal variety and $D$
an effective $\Q$-divisor on $X$ such that $K_X+D$ is $\Q$-Cartier, the pair $(X, D)$ is
log-canonical if for every birational morphism \holom{\mu}{X'}{X} from a normal variety $X'$
we can write
$$
K_{X'} = \mu^* (K_X+D) + \sum_j a_j E_j
$$
such that $a_j \geq -1$ for all $j$.

\begin{proof}[Proof of Theorem \ref{theoremisolated}]
We will treat first the case of Fano fourfolds
and in a second step apply Floris' result to prove the higher-dimensional case.

{\em a) $\dim X=4$.}
Let $Y_1$ be a general element in $|-K_X|$ and consider the restriction sequence
$$
0 \rightarrow \sO_X \rightarrow \sO_X(-K_X) \rightarrow \sO_{Y_1}(-K_X) \rightarrow 0.
$$
Since $h^1(X, \sO_X)=0$ we have a surjection
$$
H^0(X, \sO_X(-K_X)) \rightarrow H^0(Y_1, \sO_{Y_1}(-K_X)),
$$
so a general element in the linear system $|-K_X|_{Y_1}|$ is obtained by intersecting $Y_1$ with another general element $Y_2 \in |-K_X|$.
By \cite[Prop.4.2]{Kaw00} we know that for
$$
D:=Y_1 \cap Y_2 \subset |-K_X|_{Y_1}|
$$
general, the pair $(Y_1, D)$ is log-canonical. In particular $D$ is a reduced surface, so the singular locus of $D$ has dimension at most one.
Since $D$ is a complete intersection cut out by the divisors $Y_1$ and $Y_2$, we have
$$
Y_{i,\sing} \subset D_{\sing} \qquad \forall \ i=1,2.
$$
Moreover by inversion of adjunction \cite[Thm.7.5]{Kol95} the pair $(X, Y_1+Y_2)$ is log-canonical near the divisor $Y_1$.

We will now argue by contradiction and suppose that a general element in $|-K_X|$ is singular along a curve.
Take a general element $Y'$ in the pencil $<Y_1, Y_2> \subset |-K_X|$ spanned by $Y_1$ and $Y_2$.
Then we have
$$
Y' \cap Y_1 = Y_2 \cap Y_1 = D,
$$
so we see as above that
$$
Y'_{\sing} \subset D_{\sing}.
$$
Since $Y'$ varies in an infinite family with each member having a singular locus of dimension one and
$D_{\sing}$ has dimension at most one, there exists a curve $C \subset D_{\sing}$ such that every general $Y'$ is singular
along $C$. By upper semicontinuity of the multiplicity this shows that both $Y_1$ and $Y_2$ are singular along the curve $C$.

Let now \holom{\sigma}{X'}{X} be the blow-up of $X$ along $C$. Since $X$ is smooth along $C$ we have
$$
K_{X'} = \sigma^* K_X + 2 E,
$$
where $E$ is the exceptional divisor. Moreover since $Y_1$ and $Y_2$ are singular along $C$ we have
$$
\sigma^* Y_i = Y_i' + a_i E
$$
with $a_i \geq 2$ for all $i=1,2$.
Thus the pair $(X, Y_1+Y_2)$ is not log-canonical,
a contradiction.

{\em b) $\dim X$ arbitrary.}
By Theorem \ref{theoremfloris} there exists a sequence
\[
X \supsetneq Z_1 \supsetneq  Z_2 \supsetneq \ldots \supsetneq Z_{n-3}
\]
such that $Z_{i+1} \in |H|_{Z_i}|$ is a normal variety of dimension $n-i$ with at most canonical singularities.
In particular $Z_{n-3}$ is a Calabi-Yau threefold and arguing inductively we
obtain surjective maps
$$
H^0(X, H) \rightarrow H^0(Z_i, \sO_{Z_i}(H)) \qquad \forall \ i=1, \ldots, n-3.
$$
Thus $Z_{n-3}$ is obtained by intersecting $n-3$ general elements $Y_1, \ldots Y_{n-3} \in |H|$
and a general element in $|H|_{Z_{n-3}} |$ is obtained by intersecting $Z_{n-3}$
with another general element $Y_{n-2} \in |H|$.
By \cite[Prop.4.2]{Kaw00} we see that for
$$
D:=Z_{n-3} \cap Y_{n-2} \subset |H|_{Z_{n-3}}|
$$
general, the pair $(Z_{n-3}, D)$ is log-canonical.
In particular $D$ is a reduced surface, so the singular locus of $D$ has dimension at most one.
Since $D$ is a complete intersection cut out by the divisors $Y_1, \ldots, Y_{n-2}$, we have
$$
Y_{i,\sing} \subset D_{\sing} \qquad \forall \ i=1, \ldots, n-2.
$$
Moreover by repeated use of inversion of adjunction the pair $(X, \sum_{i=1}^{n-2} Y_i)$ is log-canonical near $Z_{n-3}$.

We can now conclude as above : if a general element of $|H|$ is singular along a curve, we
can use a general element $Y'$ in the pencil $<Y_{n-3}, Y_{n-2}> \subset |H|$ to show that there exists a curve $C \subset D$
such that both $Y_{n-3}$ and $Y_{n-2}$ are singular along $C$. In particular the divisor $\sum_{i=1}^{n-2} Y_i$ has multiplicity $n$
along $C$, contradicting the log-canonicity of the pair $(X, \sum_{i=1}^{n-2} Y_i)$.
\end{proof}

\begin{proof}[Proof of Theorem \ref{theoremfourfolds}]
By Theorem \ref{theoclasses} and Theorem \ref{theoremisolated}  we are left to show that $h^0(X, H) \geq n-2$ when $n\leq 5$ and $-K_X=(n-3)H$.
This is established \cite[Thm.1.1]{Flo10} for $\dim X=4,\,5$; for the convenience of the reader we prove the fourfold case:

By Kodaira vanishing we have $h^0(X, -K_X)=\chi(X, -K_X)$,
so we know by the Riemann-Roch formula that
\[
h^0(X, -K_X)= \frac{1}{6} (-K_X)^4 + \frac{1}{12} (-K_X)^2 \cdot c_2(X) + \chi(X,  \sO_X).
\]
By a recent result of Peternell \cite[Thm.1.4]{Pet08} the tangent bundle of a Fano manifold
is generically ample. Thus we can apply a theorem of Miyaoka \cite[Thm.6.1]{Miy87}
to see that the intersection product  $(-K_X)^2 \cdot c_2(X)$ is non-negative. Since $\chi(X,  \sO_X)=1$
for any Fano manifold this implies that $h^0(X, -K_X) \geq 2.$
\end{proof}

\begin{proof}[Proof of Theorem \ref{theoremshokurov}] \label{proofshokurov}
Using Reid's method (\cite[0.5]{Rei83}, cf. also \cite[4]{Flo10}) one sees that a general element $Y \in |-K_X|$
is a K3 surface with at most canonical singularities.
The proof of \cite[Prop.4.2]{Kaw00} applies verbatim to K3 surfaces, so if $A$ is an ample Cartier divisor
on a K3 surface $Y$ with at most canonical singularities and $D \in |A|$ a general element, then $(Y, D)$
is log-canonical.

Choose now
$Y_1$ and $Y_2$ general elements in $|-K_X|$ such that the pair $$(Y_1, D:=Y_1 \cap Y_2)$$ is log-canonical.
By inversion of adjunction this implies that the pair $(X, Y_1+Y_2)$ is log-canonical near $Y_1$.
Since $Y_1 \cap Y_2$ is a reduced curve, its singular locus is a union of points.
Thus if a general element in $|-K_X|$ is singular, we can argue as in
the proof of Theorem \ref{theoremisolated} to see that there exists a point $p \in D_{\sing}$
such that both $Y_1$ and $Y_2$ are singular in $p$.
But this implies that the divisor $Y_1+Y_2$ has multiplicity at least four in $p$.
Since $X$ is a smooth threefold, the pair $(X, Y_1+Y_2)$ is not log-canonical in $p$, a contradiction.
\end{proof}

\subsection{Consequences of the canonical isolated singularities property}

We will need later on  the following vanishing results (Lemma \ref{annulationdecoh}
and Lemma \ref{annulationdeforme}): Let $Y$ be a Gorenstein projective threefold
with isolated canonical singularities, such that
$H^2({Y},\mathcal{O}_{{Y}})=0$. Let $H$ be an ample line bundle on $Y$.
\begin{lemma} \label{annulationdecoh}
For any desingularization $\tau:\widetilde{Y}\rightarrow Y$ of $Y$, one has
$$H^2(\widetilde{Y},\mathcal{O}_{\widetilde{Y}})=\{0\}$$ and
$H^0(\widetilde{Y},\Omega_{\widetilde{Y}}^2)=\{0\}$.
\end{lemma}
\begin{proof}
 Indeed, the singularities of $Y$ are canonical,
hence rational. Thus $$H^2(\widetilde{Y},\mathcal{O}_{\widetilde{Y}})=H^2({Y},\mathcal{O}_{{Y}})=0.$$
 By Hodge symmetry on $\widetilde{Y}$,
the vanishing of $H^2(\widetilde{Y},\mathcal{O}_{\widetilde{Y}})$
 implies the vanishing of $H^0(\widetilde{Y},\Omega_{\widetilde{Y}}^2)$.
 \end{proof}
 We will need in the next section
 a stronger form of this vanishing result.
\begin{lemma} \label{annulationdeforme} Under the same assumptions on $Y$, for any sufficiently large and divisible $n$, and any   $\Sigma\in |nH|$ contained in $Y_{reg}$, one has
\begin{eqnarray}\label{vanishingsectionssursigma} H^0(\Sigma,{\Omega_Y^2}_{\mid \Sigma})=\{0\}.
\end{eqnarray}

\end{lemma}
\begin{proof} Let $E$ be an  effective divisor supported on the exceptional
divisor of $\tau$, such that $-E$ is $\tau$-ample.
 Note that for a smooth surface $\Sigma\in |nH|$ contained in $Y_{reg}$, $\Sigma$ can be seen
 as a surface in the linear system $|\tau^*nH|$ on  $\widetilde{Y}$ which does not intersect $E$ and
it is equivalent to show
 that for some integer $l$, one has $H^0(\Sigma,{\Omega_{\widetilde{Y}}^2(lE)}_{\mid \Sigma})=\{0\}$.
 As $-E$ is $\tau$-ample,  for $n_0$ large
enough, $n_0\tau^*H(-E)$ is ample, and thus,  by vanishing on $\widetilde{Y}$,
 we have for any $k\gg 0$
 \begin{eqnarray}\label{autrevan} H^1(\widetilde{Y}, \Omega_{\widetilde{Y}}^2(kE-kn_0\tau^*H))=0.
 \end{eqnarray}

For $n=kn_0$, let $\Sigma\in |nH|$ be contained in $Y_{reg}$, or equivalently $\Sigma\subset \widetilde{Y}$,
$\Sigma\in |\tau^*nH|$, not intersecting $E$.
 The vanishing (\ref{autrevan}) then implies that the restriction map
 $$H^0(\widetilde{Y},\Omega_{\widetilde{Y}}^2(kE))\rightarrow H^0(\Sigma,{\Omega_{\widetilde{Y}}^2(kE)}_{\mid \Sigma})$$
 is surjective.

 We now use \cite[Thm.1.4]{GKKP} which implies that
 $$H^0(\widetilde{Y},\Omega_{\widetilde{Y}}^2(kE))=H^0(\widetilde{Y},\Omega_{\widetilde{Y}}^2)$$
 for $k\geq0$, because the singularities of $Y$ are canonical.
 By Lemma \ref{annulationdecoh}, we know that the right hand side is $0$, which concludes
 the proof.
\end{proof}

\section{Application to curve classes}

This section is devoted to the proof of Theorem
 \ref{theoclasses}.  The statement  will first  be
 reduced to the following Proposition \ref{secondprop},
 a variant of Proposition 1 in \cite{voisinuniruled},  which
establishes this proposition in the case where $Y$ is  smooth. We assume here that
 $Y$ is  a Gorenstein projective threefold with trivial canonical bundle
and  canonical isolated  singularities, such that
 $H^1({Y},\mathcal{O}_{{Y}})=H^2({Y},\mathcal{O}_{{Y}})=0$. Let $H$ be an ample line bundle on $Y$.

\begin{proposition} \label{secondprop}  Under these assumptions,  there exists an integer $n_0$
such that for any multiple $n$ of $n_0$, there is a smooth
 surface $\Sigma\subset Y$ in the linear system $|nH|$ with the following property:
 The integral cohomology $H^2(\Sigma,\mathbb{Z})$ is generated over $\mathbb{Z}$ by
 classes $\beta_i$ which become algebraic on some small deformation $\Sigma_i$ of $\Sigma$ in $Y$.
\end{proposition}
In this statement,  a small deformation $\Sigma_i$ of $\Sigma$ in
 $X$  is a surface parameterized by a point in a small
ball in $|nH|$ centered at the point parameterizing $\Sigma$. The family of surfaces
parameterized by this ball is then topologically trivial, which makes the parallel transport of $\beta_i$ to a cohomology class
on $\Sigma_i$  canonical.

We postpone the proof of  this proposition, and prove now Theorem
\ref{theoclasses}.
\begin{proof}[Proof of Theorem \ref{theoclasses}]
Let $Y\subset X$ be as in the statement of Theorem \ref{theoclasses}. $Y$ is a local complete
intersection and  the canonical bundle $K_Y$ is trivial by adjunction, using the equality
${\rm det}\,\mathcal{E}=-K_X$.
As $X$ is rationally connected, we have $H^q(X,\mathcal{O}_X)=0$ for $q\geq 1$.
As $-K_X$ is $1$-ample, it follows that
\begin{eqnarray}\label{van1}H^1(Y,\mathcal{O}_Y)=0,
\end{eqnarray}
using the Sommese vanishing theorem \cite[Proposition 1.14]{sommese}  and the Koszul resolution of $\mathcal{I}_Y$.
As $Y$ has trivial canonical bundle,  we then also have
\begin{eqnarray}\label{van2}H^2(Y,\mathcal{O}_Y)=0\end{eqnarray}
 by Serre's duality.

Let $H_X$ be an ample line bundle on $X$, and $H$ its restriction to $Y$.
  Let $n$ be chosen  in such a way that
the conclusion of Proposition \ref{secondprop} holds for the pair $(Y,H)$.
Let  $Z$ be a general member of $|nH_X|$.
Consider  the complete intersection surface
$$\Sigma:=Z\cap Y\stackrel{j}{\hookrightarrow} X.$$

We claim that the Gysin
map
$$j_*:H^2(\Sigma,\mathbb{Z})=H_2(\Sigma,\mathbb{Z})\rightarrow H_2(X,\mathbb{Z} )= H^{2n-2}(X,\mathbb{Z})$$
is surjective.  Indeed, we first make the remark
that the restriction to a general sufficiently ample  hypersurface of $X$ of a $1$-ample
vector bundle  on $X$ is ample. Thus we may assume that the chosen hypersurface $Z$ is
general and that $\mathcal{E}_{\mid Z}$ is ample of rank $n-3$.
But the surface
$\Sigma\subset Z$ is the zero locus of a section
of the ample vector bundle $\mathcal{E}_{\mid Z}$ of
rank $n-3$ on $Z$, and ${\rm dim}\,Z=n-1$;  Sommese's theorem \cite[Proposition 1.16]{sommese}  extending
 Lefschetz hyperplane section theorem
 thus implies that the natural map
 $$j'_*:H^2(\Sigma,\mathbb{Z})=H_2(\Sigma,\mathbb{Z})\rightarrow H_2(Z,\mathbb{Z} )$$
 is surjective,
 where $j'$ is the inclusion of
 $\Sigma$ in $Z$.
We finally apply the Lefschetz theorem on hyperplane sections to the inclusion
$k:Z\hookrightarrow  X$
to conclude that the Gysin map
$k_*:H_{2}(Z,\mathbb{Z})\rightarrow H_{2}(X,\mathbb{Z})$ is also surjective. Hence $j_*=k_*\circ j'_*:H_2(\Sigma,\mathbb{Z})\rightarrow H_2(X,\mathbb{Z} )$ is surjective, which  proves the claim.

Let $\Sigma=Z\cap Y$ be as above
and let $\beta \in H^{2n-2}(X,\mathbb{Z})$. Then $\beta=j_*\alpha$, for some
$\alpha\in H^2(\Sigma,\mathbb{Z})$.
$Y$ has by assumption isolated canonical singularities, and the needed vanishings
(\ref{van1}) and (\ref{van2}) hold. Thus we can apply Proposition
\ref{secondprop}.
We thus can write $\beta=\sum_i\beta_i$, where
$\beta_i$ has the property that its parallel transport $\beta'_i\in H^2(\Sigma_{i},\mathbb{Z})$ is
the class $[D_i]$ of a divisor $D_i$ on $\Sigma_i$, where $j_i:\Sigma_{i}\hookrightarrow  X$ is a small deformation of $\Sigma$.
We then have
$$j_*\beta_i=j_{i*}\beta'_i=j_{i*}([D_i])$$
and thus the class $\beta$ can be written as $\sum_ij_{i*}([D_i])$, which is the class of the
$1$-cycle $\sum_ij_{i*}(D_i)$ of $X$.
\end{proof}

It remains to prove Proposition \ref{secondprop}.
As already mentioned, the case where $Y$ is smooth is done in \cite{voisinuniruled}. It is based on the study
of the infinitesimal variation of the Hodge structure on $H^2(\Sigma,\mathbb{Z})$ and is in some sense
purely local. However,
we need to point out carefully the places where we use the assumptions on the singularities
of  $Y$, as the proposition is wrong if we only assume that $H^i(Y,\mathcal{O}_Y),\,i=1,\,2$
and $Y$ is Gorenstein with  trivial canonical bundle. The simplest example is the case where
$Y$ is a quintic hypersurface in $\mathbb{P}^4$ which is a cone over a smooth quintic surface
$S$ in
$\mathbb{P}^3$. Any smooth surface $\Sigma\subset Y_{reg}$ dominates $S$
via the projection $p$ from the vertex of the cone, and only those
cohomology classes $\beta$ of degree $2$ on $\Sigma$ which satisfy the property
that $p_*\beta$ is of type $(1,1)$ on $S$ can become algebraic on a small deformation
of $\Sigma$ in $Y$. Thus these classes do not generate $H^2(\Sigma,\mathbb{Z})$ since
the Hodge structure on $H^2(S,\mathbb{Z})$ is non trivial. In this example, the singularities of $Y$ are isolated but not canonical and in fact the vanishing $H^2(\widetilde{Y},\mathcal{O}_{\widetilde{Y}})=0$ of
Lemma \ref{annulationdecoh} does not hold.

For the convenience of the reader, we first summarize the strategy of the proof
of Proposition \ref{secondprop},
used in \cite{voisinuniruled}, which mainly consists to a reduction to Proposition \ref{propIVHS} below.

Let $\Sigma\in |nH|$, and let us choose a small ball $U\subset |nH|$ centered at the point
$0$ parameterizing $\Sigma$. Let
$$\pi:\mathcal{S}\rightarrow U,\,\mathcal{S}\subset U\times Y,$$
be the restriction to $U$ of the universal family.
This is a smooth projective map, and there is a variation of Hodge structures
on the local system $R^2\pi_*\mathbb{Z}$, which is trivial since $U$ is simply connected.

Let $\mathcal{H}^2$ be the holomorphic vector bundle
 $R^2\pi_*\mathbb{Z}\otimes\mathcal{O}_U$. It is endowed with the Gauss-Manin connection
 $\nabla:\mathcal{H}^2\rightarrow\mathcal{H}^2\otimes\Omega_U$, and the Hodge filtration by holomorphic
 subbundles $F^i\mathcal{H}^2$. Furthermore, Griffiths transversality holds (cf. \cite[I, 10.1.2]{voisinbook}):
 $$\nabla(F^i\mathcal{H}^2)\subset F^{i-1}\mathcal{H}^2\otimes\Omega_U.$$

Denote by $F^1H^2$ the total space of  $F^1\mathcal{H}^2$. The trivialisation of the local system
$R^2\pi_*\mathbb{Z}$ on $U$ gives us a holomorphic map
$$\Phi: F^1H^2\rightarrow H^2(\Sigma,\mathbb{C})$$
which to a class $\alpha\in F^1H^2(\mathcal{S}_t)$ associates its parallel transport to
$\mathcal{S}_0=\Sigma$.

We make four observations :

a) First of all, for the proof of \ref{secondprop}, we can work with cohomology
$H^2(\Sigma,\mathbb{Z})/{\rm torsion}$, because the torsion is made of classes of divisors.

b) Secondly,
the set of classes $\beta_t\in H^2(\Sigma,\mathbb{Z})/{\rm torsion}$ which become algebraic (or equivalently
of type
$(1,1)$) on  $\mathcal{S}_t$ for some $t\in U$ identifies (via the inclusion of
$H^2(\Sigma,\mathbb{Z})/{\rm torsion}$ in $H^2(\Sigma,\mathbb{C})$) to
the intersection of the image of $\Phi$ with $H^2(\Sigma,\mathbb{Z})/{\rm torsion}$.

c)    Consider the real part $F^1H^2_\mathbb{R}$
of $F^1H^2$, namely its intersection
with the real vector bundle $H^2_\mathbb{R}$ with fiber $H^2(\mathcal{S}_t,\mathbb{R})$
at $t\in U$.  We clearly have $F^1H^2_\mathbb{R}=\Phi^{-1}(H^2(\Sigma,\mathbb{R}))$.

d) Let us denote by $\Phi_\mathbb{R}$ the  restriction of  $\Phi$
to $F^1H^2_\mathbb{R}$. Then $Im\,\Phi_\mathbb{R}$ is a cone.

\vspace{0.5cm}

What we want to prove is that the lattice $H^2(\Sigma,\mathbb{Z})/{\rm torsion}$
is generated over $\mathbb{Z}$ by points in $Im\,\Phi_\mathbb{R}\cap H^2(\Sigma,\mathbb{Z})/{\rm torsion}$. We use Lemma 3
in \cite{voisinuniruled}, which says that given a lattice $L$, integral points of
a cone in $L_\mathbb{R}$ with non-empty interior set generate $L$ over $\mathbb{Z}$.
Using this and the above observation d), it suffices to prove that
$Im\,\Phi_\mathbb{R}$ has non-empty interior set and by observation c) and Sard's Lemma,
it  suffices to show that
$\Phi$ is a submersion at some real point. Finally,   as the set
of points where $\Phi$ is a submersion is Zariski open, it  suffices to show that
$\Phi$ is a submersion at some  point.

We are now reduced to a statement involving the infinitesimal variations of Hodge structures on
$H^2(\mathcal{S}_t)$ thanks to the following Lemma \ref{submersionhodge} (cf. \cite[II, 5.3.4]{voisinbook}):

Using transversality, the Gauss-Manin connection induces $\mathcal{O}_U$-linear maps
$$\overline{\nabla}: F^i\mathcal{H}^2/F^{i-1}\mathcal{H}^2\rightarrow F^{i-1}\mathcal{H}^2/F^{i}\mathcal{H}^2\otimes\Omega_U$$
whose fiber at $0\in U$ gives for $i=1$:
$$\overline{\nabla}:H^1(\Sigma,\Omega_\Sigma)\rightarrow Hom\,(T_{U,0},H^2(\Sigma,\mathcal{O}_\Sigma)),$$
where $T_{U,0}= H^0(\Sigma,nH_{\vert \Sigma})$.
We will use the same notation for the  map obtained for $i=2$:
$$\overline{\nabla}:H^0(\Sigma,K_\Sigma)\rightarrow Hom\,(T_{U,0},H^1(\Sigma,\Omega_\Sigma)),$$
\begin{lemma}\label{submersionhodge} Let $\tilde{\lambda}\in F^1H^2(\Sigma,\mathbb{C})$ project to
$\lambda\in H^1(\Sigma,\Omega_\Sigma)=F^1H^2(\Sigma,\mathbb{C})/F^2H^2(\Sigma,\mathbb{C})$.
Then $\Phi$ is a submersion at $\tilde{\lambda}$ if and only if $\overline{\nabla}(\lambda)$
is a surjective homomorphism from $T_{U,0}$ to $H^2(\Sigma,\mathcal{O}_\Sigma)$.
\end{lemma}

Combining these facts, we see that Proposition \ref{secondprop} is a consequence of the following:

\begin{proposition} \label{propIVHS} Under the same  assumptions as in Proposition
\ref{secondprop},  there exists an integer $n_0$
such that for any multiple $n$ of $n_0$,  for generic
$\Sigma\in |nH|$ and for generic $\lambda\in H^1(\Sigma,\Omega_\Sigma)$, the map
$\overline{\nabla}(\lambda): H^0(\Sigma,\mathcal{O}_\Sigma(nH))\rightarrow H^2(\Sigma,\mathcal{O}_\Sigma)$ is surjective.
\end{proposition}
This proposition, in the case where $Y$ is smooth, already appeared
in \cite{voisinintjmath}, with a much simpler proof  given in \cite{voisinuniruled}.

Note that because $K_{\widetilde{Y}\mid \Sigma}$ is trivial, one has by adjunction
$$K_\Sigma\cong \mathcal{O}_\Sigma(nH)$$
and the two spaces $H^0(\Sigma,\mathcal{O}_\Sigma(nH)),\,H^2(\Sigma,\mathcal{O}_\Sigma)$
are of the same dimension, and more precisely dual to each other, the duality being
canonically determined by a choice of trivialization of $K_{\widetilde{Y}\mid \Sigma}$.

It follows from formulas (\ref{newref14sep}) and (\ref{montrelasymetrie}) given in
 Lemma  \ref{le13septembre}
 below that the homomorphisms $\overline{\nabla}(\lambda)$ are symmetric with respect
 to this duality. They can thus be seen as a system of quadrics on $H^0(\Sigma,\mathcal{O}_\Sigma(nH))$,
 and the statement is that the generic one is non singular.

 The proof of Proposition \ref{propIVHS} in the smooth case uses Griffiths' theory which computes the Hodge filtration
 on the middle  cohomology of a sufficiently ample hypersurface in a smooth projective
 variety and its infinitesimal variations, using residues of meromorphic forms. Following \cite{green}, the arguments can be made more formal and use only local infinitesimal data along the given hypersurface. But the vanishing
 needed still will use some global assumptions on the ambient variety.
In fact, the main technical point that we will need and where we will actually use the assumptions
 on the singularities of $Y$ is  Lemma \ref{surjH2} below.

 Before stating it, let us recall a few points concerning infinitesimal variations of Hodge structures of
 hypersurfaces (surfaces in our case).
 Consider    a smooth surface $\Sigma\subset Y$ in a linear system $|nH|$ on $Y$, where $H$ is ample
 on $Y$, and $Y$ is projective, smooth along $\Sigma$.
 There are two exact sequences deduced from the normal bundle sequence of $\Sigma$ in $Y$ (or rather $Y_{reg}$).
 \begin{eqnarray}\label{ex1} 0\rightarrow\Omega_\Sigma(nH)\rightarrow \Omega^2_{Y\mid\Sigma}(2nH)\rightarrow K_\Sigma(2nH)
 \rightarrow 0,
 \end{eqnarray}
 \begin{eqnarray}\label{ex2}0\rightarrow\mathcal{O}_\Sigma\rightarrow \Omega_{Y\mid\Sigma}(nH)\rightarrow \Omega_\Sigma(nH)\rightarrow 0.
 \end{eqnarray}
  Applying the long exact sequence associated to (\ref{ex1}) provides us with a map
  $$\delta_1:H^0(\Sigma,K_\Sigma(2nH))\rightarrow H^1(\Sigma,\Omega_\Sigma(nH)).$$
  The long exact sequence associated to (\ref{ex2}) provides similarly
  a map
  $$\delta_2: H^1(\Sigma,\Omega_\Sigma(nH))\rightarrow H^2(\Sigma,\mathcal{O}_\Sigma).$$
  Set $\delta:=\delta_2\circ\delta_1:H^0(\Sigma,K_\Sigma(2nH))\rightarrow H^2(\Sigma,\mathcal{O}_\Sigma)$.

  Similarly, let \begin{eqnarray}\label{ex38SEP}\delta':H^0(\Sigma,K_\Sigma(nH))\rightarrow H^1(\Sigma,\Omega_\Sigma)
 \end{eqnarray}
  be the map induced by the short exact sequence
   $$ 0\rightarrow\Omega_\Sigma\rightarrow \Omega^2_{Y\mid\Sigma}(nH)\rightarrow K_\Sigma(nH)
 \rightarrow 0.
 $$
Let us recall  the relevance of these maps to
the study of infinitesimal variations of Hodge structures of the surfaces
$\Sigma\subset Y$.

 \begin{lemma}\label{le13septembre} We have for $u\in H^0(\Sigma,\mathcal{O}_\Sigma(nH))$, $\eta\in H^0(\Sigma, K_\Sigma)$
 \begin{eqnarray} \label {newref14sep}\overline{\nabla}(\eta)(u)=\delta'(\eta u),
 \end{eqnarray}

 Furthermore, for $\sigma\in H^0(\Sigma,K_\Sigma(nH))$,  $u,\,\eta$ as above, we have
 \begin{eqnarray}\label{autre13sept} \overline{\nabla}(\delta'(\sigma))(u)=\delta(\sigma u)\,\,{\rm in}\,\,H^2(\Sigma,\mathcal{O}_\Sigma).
 \end{eqnarray}
 Finally,  for $\lambda\in H^1(\Sigma, \Omega_\Sigma)$, and any $u,\,\eta$ as above,
 \begin{eqnarray}\label{montrelasymetrie} \langle\overline{\nabla}(\lambda)(u),\eta\rangle=-\langle\lambda,\overline{\nabla}(\eta)(u)\rangle,
 \end{eqnarray}
 where the  first pairing uses Serre's duality between $H^0(\Sigma,K_\Sigma)$ and
 $H^2(\Sigma,\mathcal{O}_\Sigma)$, while the second pairing is the intersection pairing on $H^1(\Sigma,\Omega_\Sigma)$.
 \end{lemma}
 \begin{proof} The two formulas (\ref{newref14sep}) and (\ref{autre13sept}) follow from Griffiths'
 general description of
 the maps $\overline{\nabla}$ acting on  a given infinitesimal deformation
 $\rho(u)\in H^1(\Sigma,T_\Sigma)$ of $\Sigma$ (cf. \cite[I, Theorem 10.21]{voisinbook}), and from the fact that
 the exact sequences written above are all twists of the normal bundle
 exact sequence
 $$0\rightarrow T_\Sigma\rightarrow T_{Y\mid\Sigma}\rightarrow \mathcal{O}_\Sigma(nH)\rightarrow0$$
 which governs the Kodaira-Spencer map $\rho:H^0(\Sigma,\mathcal{O}_\Sigma(nH))\rightarrow H^1(\Sigma,T_\Sigma)$.

 Formula (\ref{montrelasymetrie}) is formula (5.14) proved in \cite[II, 5.3.3]{voisinbook}.
 \end{proof}

\begin{lemma} \label{surjH2} Assume that $n$ is sufficiently large  and that $\Sigma$ satisfies
\begin{eqnarray}\label{hypoth}H^0(\Sigma,\Omega_{Y\mid \Sigma}^2)=0.\end{eqnarray}
Then $\delta$ is surjective.

In particular, under the assumptions of Proposition \ref{secondprop}, the map
$\delta$ is surjective for $n$ sufficiently divisible.
\end{lemma}

\begin{proof}
The second statement follows indeed from the first, since we proved in Lemma \ref{annulationdeforme}
that when $Y$  has isolated canonical singularities  and
$H^2({Y},\mathcal{O}_{{Y}})=0$,  we have  the vanishing (\ref{hypoth})
for $n$ sufficiently divisible.

We first prove  that $\delta_1$ is surjective for large $n$ and for any $\Sigma$ as above, without any further
assumption on $Y$. Indeed, looking at the long exact sequence associated to (\ref{ex1}), we find that the cokernel
of $\delta_1$ is contained in $H^1(\Sigma,\Omega^2_{Y\mid\Sigma}(2nH))$.  Consider the exact sequence:
$$0\rightarrow \Omega^2_{Y}(nH)\rightarrow \Omega^2_{Y}(2nH)\rightarrow\Omega^2_{Y\mid\Sigma}(2nH)\rightarrow 0.$$

We get by Serre's vanishing on $Y$ that for large $n$, both
$H^1(Y, \Omega^2_{Y}(2nH))$ and $H^2(Y, \Omega^2_{Y}(nH))$ vanish  and it follows that
$H^1(\Sigma,\Omega^2_{Y\mid\Sigma}(2nH))$
also vanishes.

It remains to prove that $\delta_2$ is surjective under assumption (\ref{hypoth}), for large $n$.
Looking at the long exact sequence associated to (\ref{ex2}), we find that the cokernel
of $\delta_2$ is contained in $H^2(\Sigma,\Omega_{Y\mid\Sigma}(nH))$. As
$K_\Sigma=nH|_\Sigma$ by adjunction, we find that
$H^2(\Sigma,\Omega_{Y\mid\Sigma}(nH))$ is Serre dual to
$$H^0(\Sigma,T_{Y\mid\Sigma}(K_{Y\mid \Sigma}))=H^0(\Sigma,\Omega_{Y\mid\Sigma}^2),$$
which vanishes by assumption.
\end{proof}
For the convenience of the reader, we now summarize the main steps in the
proof of Proposition \ref{propIVHS} given in \cite{voisinuniruled} in the case where
$Y$ is smooth with trivial canonical bundle, in order to make clear why
 the argument still works in the singular case, when $Y$ has isolated canonical singularities.

 Let $n$  be divisible enough so that the
 conclusion of Lemma \ref{annulationdeforme} holds. For any smooth $\Sigma\in |nH|$, we
  have the system $Q_\Sigma$
 of quadrics on
 $H^0(\Sigma,K_\Sigma)\cong H^0(\Sigma,\mathcal{O}_\Sigma(nH))$ given
 as the image of $H^1(\Sigma,\Omega_\Sigma)$ in $S^2H^0(\Sigma,\mathcal{O}_\Sigma(nH))^*$,
 via the map \begin{eqnarray}\label{definq}q,\, \lambda\mapsto q_\lambda,
 \end{eqnarray} where
 $$ q_\lambda(\eta,u):=<\lambda,\overline{\nabla}(\eta)(u)>,
 $$
 for $\eta\in H^0(\Sigma,K_\Sigma)=H^0(\Sigma,\mathcal{O}_\Sigma(nH)),\,u\in H^0(\Sigma,\mathcal{O}_\Sigma(nH))$.

 As explained above, the statement to be proved can be restated by
 saying that for generic $\Sigma$, $Q_\Sigma$ contains a smooth quadric.

 We will apply the following lemma (cf. \cite[Proposition 1.10]{voisinuniruled}, \cite[Lemma 15]{voisinintjmath}) which is an easy consequence
 of Bertini's Lemma:
 \begin{lemma} \label{lecriterion} A linear system of quadrics $Q$ on a vector space $V$ contains a
 smooth member if the following  condition (*) holds:

 (*) Let $B\subset \mathbb{P}(V)$ be the base locus of $Q$ and $b:={\rm dim}\,B$. Then
 for any $k\leq b$, the subset
 $ W_k\subset \mathbb{P}(V)$, defined as
 $$W_k:=\{v\in B,\,i_v:Q\rightarrow V^*\,\,{\rm has \,\,rank}\,\,\leq k\}$$
 has dimension $< k$.
 \end{lemma}
 (In this statement, the map $i_v$ associated to $v$ is constructed by contraction.)
 For example, if $b=0$, the base-locus consists of isolated points, and by Bertini, if any member of
 $Q$ was singular, it would be singular at some common point $v\in B$, which would
 satisfy $i_v=0$, which is excluded by assumption (*).

 The first step in the proof of the proposition is the following asymptotic upper-bound
 for the dimension of the base-locus of $Q_\Sigma$ for generic $\Sigma\in |nH|$
 as a function of $n$.

 \begin{lemma} \label{lebaselocus} There exists  a positive number $c$, such that for any $n$ and for
 generic $\Sigma\in |nH|$, the dimension of the base-locus of $Q_\Sigma$ is
 $\leq cn^2$.
 \end{lemma}
 The proof of this lemma is based on the following arguments, none involving the smoothness of $Y$ away from
 $\Sigma$.

 The first argument is Proposition 1.6 in \cite{voisinintjmath}, which concerns degenerations of $\Sigma$
 to a surface $\Sigma_0$ with nodes $p_1,\ldots , p_N$: that is we have a smooth family of surfaces
 $\mathcal{S}\rightarrow \Delta$, $\mathcal{S}\subset \Delta\times Y$, with  central fiber isomorphic to $\Sigma_0$.
 \begin{lemma} \label{lelimit}The limiting linear system $Q_{\Sigma_0} $ of quadrics on
 $V_{\Sigma_0} $ contains the quadrics $q_i$ defined (up to a coefficient depending on the trivialisation of $\mathcal{O}_{p_i}(2nH)$) by
 $$q_i(\eta)=\eta^2(p_i).$$

 \end{lemma}
The second argument is then provided by the construction of surfaces $\Sigma_0$
with many nodes, imposing many conditions on
$H^0(\Sigma_0,\mathcal{O}_\Sigma(nH))$. Such surfaces $\Sigma_0$ are obtained
in \cite{voisinuniruled}  as discriminant loci for symmetric matrices of size
$(n,n)$ with coefficients in $H^0(Y,H)$ (where $H$ is supposed to be very ample on $Y$).
Using results of Barth in \cite{barth},  such a generically chosen discriminant
surface has only ordinary double points, whose number $N$
 is a cubic polynomial in $n$, (precisely equal to $\begin{pmatrix}{n+1}\\{3}\end{pmatrix}H^3$), hence with
leading term $\frac{n^3}{6}H^3$. Notice that, by Riemann-Roch,
 $V_{\Sigma_0}:=H^0(\Sigma_0,\mathcal{O}_{\Sigma_0}(nH))$ has dimension given by a cubic polynomial in
 $n$ with leading term also equal to
 $ \frac{n^3}{6} H^3$.  Finally, we prove
 in \cite{voisinuniruled}, by  considering a natural resolution
  of the ideal sheaf of subset $W=\{p_1,\ldots,p_N\}\subset \Sigma_0$ of nodes of $\Sigma_0$,
  that $W$ imposes independent conditions to
 the linear system $\mid \mathcal{O}_Y((n+2)H) \mid$, and it follows easily that for some constant
 $c>0$
 \begin{eqnarray}\label{ineg13sept}{\rm dim}\,\,H^0(Y,\mathcal{I}_W(nH))\leq cn^2.
 \end{eqnarray}
The proofs given there do not involve the smoothness of $Y$, as long as $\Sigma_0$ does not meet
the singular locus of $Y$.

 By Lemma \ref{lelimit}, the base locus
 of the limiting system of quadrics
 $Q_{\Sigma_0}$ on $V_{\Sigma_0}$ is contained in $\mathbb{P}(H^0(\Sigma_0,\mathcal{I}_W(nH)))$
 so that (\ref{ineg13sept})
  completes the proof of Lemma \ref{lebaselocus}.

 We want next to study the sets $W_{\Sigma,k}$ introduced in Lemma \ref{lecriterion}, for generic
 $\Sigma\in|nH|$. By
 Lemma \ref{lebaselocus}, we can restrict to the range $k\leq cn^2$.
  The  second step in the proof is thus the study of the locus
  $ W_{c,\Sigma}\subset\mathbb{P}(V_\sigma)$ defined as
  $$ W_{cn^2,\Sigma}:=\{v\in \mathbb{P}(V_\sigma),\,i_v:Q_\Sigma\rightarrow V^*\,\,{\rm has \,\,rank}\,\,\leq cn^2\}$$
  and we need to prove the following lemma (cf. Lemma 12 in \cite{voisinuniruled}):
  \begin{lemma}\label{rangcn2} There exists a constant $A$ independent of $n$, such that for generic
  $\Sigma\in |nH|$, one has ${\rm dim}\, W_{cn^2,\Sigma}\leq A$.
  \end{lemma}
  The proof of this statement given in {\it loc. cit.}  does not involve the smoothness
  of the ambient space $Y$. It is obtained by studying the case where
  $\Sigma$ is of Fermat type. Namely, one chooses a generic projection
  $p:Y\rightarrow \mathbb{P}^3$ such that $p^*\mathcal{O}_{\mathbb{P}^3}(1)=H$, and
 one takes for $\Sigma$ the inverse image of the Fermat surface $S_n$ of degree $n$ in
  $\mathbb{P}^3$. One is led to compare the infinitesimal variation of Hodge structure for
  $\Sigma$ with a sum of twisted infinitesimal variations of Hodge structure for $S_n$.
  This comparison involves only the map $p:\Sigma\rightarrow S_n$ and the restriction of
  $\Omega_Y$ to $\Sigma$, and not the geometry of $Y$ away from
  $\Sigma$.

  Using Lemma  \ref{rangcn2}, we conclude that all the sets
   $Z_{k,\Sigma}$ introduced in Lemma
  \ref{lecriterion} have dimension
  $\leq A$, and  in order to check the criterion given in Lemma
  \ref{lecriterion}, it suffices now to study the following sets $W_{A,\Sigma}$:
  $$W_{A,\Sigma}:=\{v\in \mathbb{P}(V_\sigma),\,i_v:Q_\Sigma\rightarrow V^*\,\,{\rm has \,\,rank}\,\,\leq cn^2\}.$$
  This is the third step and last step of the proof:
  \begin{lemma}\label{lelast} Let $H$ be normally generated on $Y$ and let $A$ be
  a given (large) constant. Then for $n>A$, and for
  smooth $\Sigma\in |nH|$, the
  set $W_{A,\Sigma}$ is empty.
  \end{lemma}
  \begin{proof} The proof is by contradiction. Let $v\in W_{A,\Sigma}$ and let
  $M\subset H^0(\Sigma,\mathcal{O}_\Sigma(2nH))$ be the kernel
  of the composite map:
  $$H^0(\Sigma,\mathcal{O}_\Sigma(2nH))\stackrel{\delta'}{\rightarrow} H^1(\Sigma,\Omega_\Sigma)
  \stackrel{i_v\circ q}{\rightarrow} H^0(\Sigma,K_\Sigma)^*,$$
  where the map
  $\delta':H^0(\Sigma,\mathcal{O}_\Sigma(2nH)){\rightarrow} H^1(\Sigma,\Omega_\Sigma)$ is defined in (\ref{ex38SEP})
  and $q$ is defined in (\ref{definq}). As the rank of $i_v\circ q$ is $\leq A$, one has   ${\rm codim}\,M\leq A$.
 Identifying $H^0(\Sigma,K_\Sigma)^*$ with $H^2(\Sigma,\mathcal{O}_\Sigma)$, and using formula (\ref{autre13sept}) of
  Lemma \ref{le13septembre}, one sees on the other hand
 that $M$ is the kernel of the composite map:
 $$H^0(\Sigma,\mathcal{O}_\Sigma(2nH))\stackrel{v\cdot}{\rightarrow}
 H^0(\Sigma,\mathcal{O}_\Sigma(3nH))\stackrel{\delta}{\rightarrow} H^2(\Sigma,\mathcal{O}_\Sigma),$$
 where $v\cdot$ is multiplication by $v$.

 It  also follows from the combination
 of the three formulas given in Lemma \ref{le13septembre} that for $v,\,v'\in H^0(\Sigma,\mathcal{O}_\Sigma(nH))$,
 and for any $m\in H^0(\Sigma,\mathcal{O}_\Sigma(2nH))$, we have
 \begin{eqnarray}\label{ultime}\langle v,\delta(v'm)\rangle=\langle v',\delta(vm)\rangle,
 \end{eqnarray}
 where $\langle\,,\,\rangle$ is the Serre pairing between
 $H^0(\Sigma,\mathcal{O}_\Sigma(nH))=H^0(\Sigma,K_\Sigma)$ and $H^2(\Sigma,\mathcal{O}_\Sigma)$.
 Taking $m\in M$, so that $\delta(mv)=0$, we deduce from (\ref{ultime}) that
 $\delta(H^0(\Sigma,\mathcal{O}_\Sigma(nH))\cdot M)$ is orthogonal to $v$ with respect to
 Serre duality.
 But this provides a contradiction by the following argument:

 First of all, $M$ has no base point. Indeed, from the exact sequence
 (\ref{ex38SEP}), we see that
  $M$ contains the image of $H^0(\Sigma,\Omega^2_{Y\mid\Sigma}(nH))$ in $H^0(\Sigma,\mathcal{O}_\Sigma(2nH))$. As $A$ is large and $2n>A$,
 $\Omega^2_{Y\mid\Sigma}(nH))$ is globally generated and its image in $H^0(\Sigma,\mathcal{O}_\Sigma(2nH))$
 thus generates $\mathcal{O}_\Sigma(2nH)$ at any point.

 Thus we can apply the following result due to Green \cite{green1} to the linear system $M$ on $\Sigma$.
\begin{proposition} \label{green}
Let $Z$ be any projective manifold and $H$ be a very ample  normally generated line
bundle on $Z$. Let $A$ be a given constant, and for $m>A$,
let $K \subset H^0(Z,\mathcal{O}_Z(mH))$ be a subspace of
codimension $\leq A$. Then
$$H^0(Z,\mathcal{O}_Z(H))\cdot K\subset H^0(Z,\mathcal{O}_Z((m+1)H))$$ has codimension $\leq A$,
with strict inequality if $K$ has no base-point and $K\not=H^0(Z,\mathcal{O}_Z(H))$.
\end{proposition}
As $M$ has no base point and $2n>A$, we conclude that the codimensions of
the subspaces $H^0(\Sigma,\mathcal{O}_\Sigma(kH))\cdot M\subset H^0(\Sigma,\mathcal{O}_\Sigma((2n+k)H))$
are strictly decreasing until they fill-in $H^0(\Sigma,\mathcal{O}_\Sigma((2n+k)H))$. As ${\rm codim}\,M\leq A<n$, we conclude that $H^0(\Sigma,\mathcal{O}_\Sigma(nH))\cdot M= H^0(\Sigma,\mathcal{O}_\Sigma(3nH))$.

But then we use Lemma \ref{surjH2} saying that $\delta$ is surjective. It thus follows that
$$\delta(H^0(\Sigma,\mathcal{O}_\Sigma(nH))\cdot M)=H^2(\Sigma,\mathcal{O}_\Sigma),$$
which
contradicts the fact that $\delta(H^0(\Sigma,\mathcal{O}_\Sigma(nH))\cdot M)$ is orthogonal to $v$ with respect to
 Serre duality.
\end{proof}

\end{document}